\documentclass[14pt]{article}

\usepackage{mathtext}
\usepackage[T1,T2A]{fontenc}
\usepackage[utf8]{inputenc}
\usepackage[english]{babel}

\usepackage{amsfonts, amsmath, amssymb, amsthm, amscd}
\usepackage{mathtools}
\usepackage{xcolor}
\usepackage{ulem}
\usepackage{geometry}
\usepackage{graphicx}
\usepackage{hyperref}
\usepackage{cancel}
\hypersetup{pdfstartview=FitH,  linkcolor=blue,urlcolor=urlcolor, citecolor=blue, colorlinks=true}

\DeclarePairedDelimiter\abs{\lvert}{\rvert}
\newcommand{\norm}[1]{\left\lVert#1\right\rVert}

\DeclareMathOperator{\supp}{supp}

\theoremstyle{theorem}
\newtheorem{theorem}{Theorem}[section]
\theoremstyle{theorem}
\newtheorem{proposition}{Proposition}[section]
\theoremstyle{definition}

\theoremstyle{lemma}
\newtheorem{lemma}{Lemma}[section]
\newcommand{\Label}[1]{\label{#1}}

\usepackage{fancyhdr}
\numberwithin{equation}{section}

\title{Fundamental Solution for a New Class of Non-Archimedean Pseudo-Differential Equations}
\author{\textbf{Anatoly N. Kochubei}\\
\footnotesize Institute of Mathematics,\\
\footnotesize National Academy of Sciences of Ukraine,\\
\footnotesize Tereshchenkivska 3, Kyiv, 01024 Ukraine,\\
\footnotesize E-mail: kochubei@imath.kiev.ua
\and
\textbf{Mariia V. Serdiuk}\\
\footnotesize Institute of Mathematics,\\
\footnotesize National Academy of Sciences of Ukraine,\\
\footnotesize Tereshchenkivska 3, Kyiv, 01024 Ukraine,\\
\footnotesize E-mail: mariia.v.serdiuk@gmail.com }

\begin{document}
\maketitle

\section{Introduction}

In the analysis of real- and complex-valued functions over a non-Archimedean field of $p$-adic numbers $\mathbb{Q}_p$ and the spaces $\mathbb{Q}_p^n$, a typical linear operator is the pseudo-differential one,
$$(Au)(x_1,\dots,x_n)=\mathcal{F}^{-1}\left[a(\xi_1,\dots,\xi_n)(\mathcal{F}u)(\xi_1,\dots,\xi_n)\right](x_1,\dots,x_n),\;x_1,\dots,x_n\in\mathbb{Q}_p,$$
where $\mathcal{F}$ is the Fourier transform. Under some conditions upon the symbol, the Fourier transform can be eliminated, and we get a representation of $A$ as a hypersingular integral operator.

The example, which has been investigated in the greatest detail, is the Vladimirov-Taibleson operator $D^{\alpha,n}$, for which
$$a(\xi_1,\dots,\xi_n)=\left(\max\left\{|\xi_1|_p,\dots,|\xi_n|_p\right\}\right)^{\alpha},\;\alpha>0,$$
and which can be seen as a kind of an elliptic operator. See \cite{VVZ,K2023}; see also \cite{Z,G-C} for equations with various related operators, such as analogs of the nonstationary Schrödinger equation, the Klein-Gordon equation, the operator with the symbol $\max\left\{|\xi|_p^{d_1},|\xi|_p^{d_2}\right\}^{\alpha}$, and others.

Note that the maximum signs in the above symbols correspond to the ultrametric geometry of the non-Archimedean spaces $\mathbb{Q}_p^n$. See Section 2 for the details.

Interesting and non-trivial equations of this theory are not necessarily formulated directly in terms of the above operators. Thus, $p$-adic parabolic equations are those with a real positive time variable and $p$-adic spatial variables; this class of equations is important for probabilistic applications. A kind of $p$-adic wave equation \cite{K2008} has the form $D_t^{\alpha}F-D_x^{\alpha,n}F=0$.

In this paper, we consider a more general equation
$$D_t^{\alpha}u(t,x)-D_x^{\beta,n}u(t,x)=0,\;\beta=K\alpha\mbox{ for some }K\in\mathbb{N}.$$
Using some techniques from \cite{K2008}, especially the idea of a radial time variable (see also \cite{K2014}), as well as the Lizorkin classes of test functions and distributions, we prove the existence and uniqueness results for the Cauchy problem, including the finite dependence property (resembling classical hyperbolic equations) and an $L^1$-estimate of the solution.

\section{Preliminaries}

\subsection{The Field of $p$-adic Numbers}

Let $p$ be a prime number. The $p$-adic norm on the field $\mathbb{Q}$ of rational numbers is defined as
$$|0|_p=0,\;|x|_p=p^{-\gamma(x)},$$
where the integer $\gamma(x)$ is obtained from the decomposition
$$x=p^{\gamma}\frac{m}{n},\;m,n,\gamma=\gamma(x)\in\mathbb{Z},$$
and the integers $m$, $n$ are prime to $p$. This norm is non-Archimedean and satisfies the strong triangle inequality $|x+y|_p\leq\max\left(|x|_p,|y|_p\right)$. The completion of the field $\mathbb{Q}$ with respect to the $p$-adic norm is called the field of $p$-adic numbers $\mathbb{Q}_p$.

Any $p$-adic number $x\neq0$ has a unique canonical representation
\begin{equation}\Label{000}
x=p^{\gamma}(x_0+x_1p+x_2p^2+\dots),
\end{equation}
where $\gamma=\gamma(x)\in\mathbb{Z}$, $x_j$ are integers such that $0\leq x_j\leq p-1$, $x_0>0$, $j=0,1,\dots$, and $|x|_p=p^{-\gamma}$.

The fractional part of a $p$-adic number $x\in\mathbb{Q}_p$ with the canonical representation \eqref{000} is defined as
\begin{equation}\Label{001}
\{x\}_p=\left\{
\begin{array}{ll}
0,\;&\mbox{if }\gamma(x)\geq0\mbox{ or }x=0;\\
p^{\gamma}(x_0+x_1p+\dots+x_{|\gamma|-1}p^{|\gamma-1|}),\;&\mbox{if }\gamma(x)<0.\\
\end{array}
\right.
\end{equation}

A function $\varphi:\mathbb{Q}_p\to\mathbb{C}$ is called {\it locally constant} if there exists such an integer $\ell$ that for any $x\in\mathbb{Q}_p$
$$\varphi(x+x')=\varphi(x),\mbox{ whenever }|x'|_p\leq p^{-\ell}.$$
The smallest of such numbers $\ell$ is called the exponent of local constancy of the function $\varphi$.

We denote by $\mathcal{D}(\mathbb{Q}_p)$ the linear topological space of all locally constant functions $\varphi:\mathbb{Q}_p\to\mathbb{C}$ with compact supports. The strong conjugate space $\mathcal{D}'(\mathbb{Q}_p)$ is called the space of Bruhat-Schwartz distributions (see \cite{A}).

\subsection{Fourier Transform}

Let $\chi_p(\xi x)=\exp(2\pi i\{\xi x\}_p)$, where $\xi\in\mathbb{Q}_p$, be an additive character of $\mathbb{Q}_p$. The Fourier transform of a test function $\varphi\in\mathcal{D}(\mathbb{Q}_p)$ is defined as
$$\widetilde{\varphi}(\xi)=\mathcal{F}_{x\to\xi}[\varphi](\xi)=\int\limits_{\mathbb{Q}_p}\chi_p(\xi x)\varphi(x)dx,\;\xi\in\mathbb{Q}_p,$$
where the integration is with respect to the standard Haar measure $dx$ on $\mathbb{Q}_p$, i.e. the invariant under shifts positive measure $dx$ such that $\int_{|x|_p\leq1}=1$.

For $n\geq1$, the Fourier transform of a test function $\varphi\in\mathcal{D}(\mathbb{Q}_p^n)$ is defined as
$$\widetilde{\varphi}(\xi)=\mathcal{F}_{x\to\xi}[\varphi](\xi)=\int\limits_{\mathbb{Q}_p^n}\chi_p(\xi\cdot x)\varphi(x)dx,\;\xi\in\mathbb{Q}_p^n,$$
where $\chi_p(\xi\cdot x)=\chi_p(\xi_1 x_1)\cdots\chi_p(\xi_n x_n)=\exp\left(2\pi i\sum\limits_{j=1}^n\{\xi_j x_j\}_p\right)$, $\xi\cdot x$ is the scalar product of vectors $\xi=(\xi_1,\dots,\xi_n)$ and $x=(x_1,\dots,x_n)$, and the function $\chi_p(\xi_jx_j)=\exp\left(2\pi i\{\xi_j x_j\}_p\right)$ for every fixed $\xi_j\in\mathbb{Q}_p$ is an additive character of the field $\mathbb{Q}_p$, $\{\xi_jx_j\}_p$ is the fractional part of a number $\xi_jx_j$ defined by \eqref{001}, $j=1,\dots,n$.

The Fourier transform $\varphi\mapsto\widetilde{\varphi}$ is a linear isomorphism from $\mathcal{D}(\mathbb{Q}_p)$ on $\mathcal{D}(\mathbb{Q}_p)$ and the inversion formula holds
$$\varphi(x)=\int\limits_{\mathbb{Q}_p}\chi_p(-x\xi)\widetilde{\varphi}(\xi)d\xi=\mathcal{F}_{\xi\to x}\left[\widetilde{\varphi}(-x)\right]=\mathcal{F}_{\xi\to x}^{-1}\left[\widetilde{\varphi}\right](x),\;x\in\mathbb{Q}_p,\;\varphi\in\mathcal{D}(\mathbb{Q}_p).$$

For a distribution $f\in\mathcal{D}'(\mathbb{Q}_p)$ its Fourier transform is defined by the formula
$$(\widetilde{f},\varphi)=(f,\widetilde{\varphi}),\;\varphi\in\mathcal{D}(\mathbb{Q}_p),$$
and the inversion formula holds
$$f=\mathcal{F}\left[\widetilde{f}(-\xi)\right],\;f\in\mathcal{D}'(\mathbb{Q}_p).$$

Similarly, for $f\in\mathcal{D}'(\mathbb{Q}_p^n)$,
$$(\widetilde{f},\varphi)=(f,\widetilde{\varphi}),\;\varphi\in\mathcal{D}(\mathbb{Q}_p^n).$$

We will use the following subspaces of $\mathcal{D}(\mathbb{Q}_p^n)$,
$$\Psi(\mathbb{Q}_p^n)=\{\psi\in\mathcal{D}(\mathbb{Q}_p^n):\;\psi(0)=0\},$$
$$\Phi(\mathbb{Q}_p^n)=\bigg\{\varphi\in\mathcal{D}(\mathbb{Q}_p^n):\int\limits_{\mathbb{Q}_p^n}\varphi(x)dx=0\;\bigg\},$$
introduced in \cite{A}. The space $\Phi(\mathbb{Q}_p^n)$ is called the Lizorkin space of test functions of the second kind. The conjugate space $\Phi'(\mathbb{Q}_p^n)$ is called the Lizorkin space of distributions of the second kind. The most important property of these spaces is that the Fourier transform $\mathcal{F}$ is a linear isomorphism from $\Psi(\mathbb{Q}_p^n)$ onto $\Phi(\mathbb{Q}_p^n)$, thus also from $\Phi'(\mathbb{Q}_p^n)$ onto $\Psi'(\mathbb{Q}_p^n)$. At the same time, $\mathcal{F}$ can be considered as a linear isomorphism from $\Phi(\mathbb{Q}_p^n)$ to $\Psi(\mathbb{Q}_p^n)$.

\subsection{Operator $D^{\alpha}$}

On a test function $\varphi\in\mathcal{D}(\mathbb{Q}_p)$, the fractional differentiation operator $D^{\alpha}$, $\alpha>0$, is defined as (see \cite{VVZ})
\begin{equation}\Label{1-6}
(D^{\alpha}\varphi)(x)=\mathcal{F}_{\xi\to x}^{-1}\big[|\xi|_p^{\alpha}(\mathcal{F}_{y\to\xi}(\varphi))(\xi)\big](x).
\end{equation}

For $x=(x_1,\dots,x_n)\in\mathbb{Q}_p^n$, define $|x|_p=\max\limits_{1\leq j\leq n}|x_j|_p$. The operator $D^{\alpha,n}$ on a function $\varphi\in\mathcal{D}(\mathbb{Q}_p^n)$ is given by the expression
$$(D^{\alpha,n}\varphi)(x)=\mathcal{F}_{\xi\to x}^{-1}\big[|\xi|_p^{\alpha}(\mathcal{F}_{y\to\xi}(\varphi))(\xi)\big](x).$$

The operator $D^{\alpha}$ can also be represented as a hypersingular integral operator. For a function $u\in\mathcal{D}(\mathbb{Q}_p)$,
\begin{equation}\Label{1-7}
(D^{\alpha}u)(x)=\frac{1-p^{\alpha}}{1-p^{-\alpha-1}}\int\limits_{\mathbb{Q}_p}|y|_p^{-\alpha-1}\left[u(x-y)-u(x)\right]dy.
\end{equation}

Similarly, if $u\in\mathcal{D}(\mathbb{Q}_p^n)$, then
$$(D^{\alpha,n}u)(x)=\frac{1-p^{\alpha}}{1-p^{-\alpha-n}}\int\limits_{\mathbb{Q}_p^n}|y|_p^{-\alpha-n}\left[u(x-y)-u(x)\right]dy.$$

Further, we shall use the following integration formulas \cite{VVZ}:
\begin{align}
\Label{0-44}
&\int\limits_{B_{\gamma}^n}d^nx=p^{n\gamma};\;\int\limits_{S_{\gamma}^n}dx=\left(1-p^{-n}\right)p^{n\gamma},\\
\Label{3.1}
&\int\limits_{B_{\gamma}^n}\chi_p(\xi x)dx=
\left\{
\begin{array}{ll}
p^{n\gamma},\;&|\xi|_p\leq p^{-\gamma};\\
0,\;&|\xi|_p\geq p^{-\gamma+1}.
\end{array}
\right.\\
\Label{3.2}
&\int\limits_{S_{\gamma}^n}\chi_p(\xi x)dx=
\left\{
\begin{array}{ll}
p^{n\gamma}(1-p^{-n}),\;&|\xi|_p\leq p^{-\gamma};\\
-p^{n(\gamma-1)},\;&|\xi|_p=p^{-\gamma+1};\\
0,\;&|\xi|_p\geq p^{-\gamma+2}.
\end{array}
\right.
\end{align}
where we denote by $B_\gamma^n\coloneqq\{x\in\mathbb{Q}_p^n:\;|x|_p\leq p^{\gamma}\}$, $S_{\gamma}^n\coloneqq\{x\in\mathbb{Q}_p^n:\;|x|_p=p^{\gamma}\}$.

\section{Radial Eigenfunctions}

Let $\alpha>0$, $\beta>0$. Consider the eigenvalue problem
\begin{equation}\Label{3.1'}
D^{\alpha}u=\lambda{u},\;\lambda=p^{\beta N},\;N\in\mathbb{Z},
\end{equation}
where $u:\mathbb{Q}_p\to\mathbb{C}$ is not identically zero.

We also suppose that
\begin{equation}\Label{2.5}
\beta=K\alpha\mbox{ for some }K\in\mathbb{N}.
\end{equation}

\begin{proposition}\Label{Pr1}
If the condition \eqref{2.5} holds, the equation \eqref{3.1'} has the set of solutions in $\Phi(\mathbb{Q}_p)$ of the following form for $N\in\mathbb{Z}$:
\begin{equation}\Label{3.12}
u_N(t)=\left\{
\begin{array}{ll}
C_Np^{KN}(1-\dfrac{1}{p}),\;&|t|_p\leq p^{-KN};\\
-C_Np^{KN-1},\;&|t|_p=p^{-KN+1};\\
0,\;&|t|_p\geq p^{-KN+2}.
\end{array}
\right.
\end{equation}
\end{proposition}

\begin{proof}
We apply the Fourier transform with respect to the both sides of the equation \eqref{3.1'} and obtain
\begin{equation}\Label{9}
(|\eta|_p^{\alpha}-p^{\beta N})(\mathcal{F}_{t\to\eta}u)(\eta)=0,\mbox{for all }\eta\in\mathbb{Q}_p.
\end{equation}
Denote $\widetilde{u}(\eta)=(\mathcal{F}_{t\to\eta}u)(\eta)$, $(\eta)\in\mathbb{Q}_p$. It follows from \eqref{9} that the equality $\widetilde{u}(\eta)\equiv0$ will not hold only if the condition $|\eta|_p^{\alpha}=p^{\beta N}$ holds, i.e. $|\eta|_p=p^{KN}$, where $K\in\mathbb{N}$ is such that $\beta=K\alpha$. This condition holds by the assumption \eqref{2.5}. (Note that if \eqref{2.5} does not hold, then the equation \eqref{3.1'} has only a zero solution.)

Since $u$ is a radial function with respect to $t$, then $\widetilde{u}$ is also a radial function with respect to $\eta$ (see \cite{K2001} formula (1.28), p.15). Therefore
\begin{equation}\Label{11}
\widetilde{u}_N(|\eta|_p)=\left\{
\begin{array}{ll}
C_N,\;&|\eta|=p^{KN};\\
0,\;&|\eta|\neq p^{KN},
\end{array}
\right.
\;C_N\neq0.
\end{equation}
By the Fourier inversion formula and the well-known integration formula \eqref{3.2}, we have
\begin{equation}\Label{12}
\begin{split}
u_N(|t|_p)=(\mathcal{F}_{\eta\to t}^{-1}\widetilde{u}_N)(t)=\int\limits_{\mathbb{Q}_p}\chi_p(-t\eta)\widetilde{u}_N(\eta)d\eta=\int\limits_{S_{KN}}\chi_p(-t\eta)C_Nd\eta=\\
=\left\{
\begin{array}{ll}
C_N(\xi)p^{KN}(1-\dfrac{1}{p}),\;&|t|_p\leq p^{-KN};\\
-C_N(\xi)p^{KN-1},\;&|t|_p=p^{-KN+1};\\
0,\;&|t|_p\geq p^{-KN+2}.
\end{array}
\right.
\end{split}
\end{equation}

It is easy to verify from \eqref{12} that $u_N\in\Phi(\mathbb{Q}_p)$ for each $N\in\mathbb{Z}$.
\end{proof}

\begin{proposition}\Label{Pr2}
If a radial distribution $u\in\Phi'(\mathbb{Q}_p)$ satisfies the equation \eqref{3.1'}, then it coincides, for some $C_N\in\mathbb{C}$, with the function \eqref{3.12}.
\end{proposition}

\begin{proof}
The proof is identical to the one given for Proposition 1 in \cite{K2008}.
\end{proof}

\section{Cauchy Problem}

Let $\alpha>0$, $\beta>0$, $n\geq1$. We consider the Cauchy problem
\begin{equation}\Label{1}
D_{|t|_p}^{\alpha}u(|t|_p,x)-D_x^{\beta,n}u(|t|_p,x)=0,\;(t,x)\in\mathbb{Q}_p\times\mathbb{Q}_p^n,
\end{equation}
\begin{equation}\Label{2}
u(0,x)=u_0(x),\;x\in\mathbb{Q}_p^n,
\end{equation}
where $u:\mathbb{Q}_p\times\mathbb{Q}_p^n\to\mathbb{C}$ is a radial function with respect to $t$.

\begin{theorem}\Label{th1}
Let $\alpha>0$, $\beta>0$ such that the condition \eqref{2.5} holds. Suppose that the function $u_0$ is in $\Phi(\mathbb{Q}_p^n)$. Then the Cauchy problem \eqref{1}-\eqref{2} has a solution $u=u(|t|_p,x)$, radial in $t$, that belongs to the space $\Phi(\mathbb{Q}_p)$ for each $x\in\mathbb{Q}_p^n$, and belongs to $\Phi(\mathbb{Q}_p^n)$ for each $t\in\mathbb{Q}_p$.

If the condition \eqref{2.5} does not hold, then the equation \eqref{1} has only a zero solution $u(t,x)\equiv0$, $(t,x)\in\mathbb{Q}_p\times\mathbb{Q}_p^n$.
\end{theorem}

\begin{proof}
Suppose that $u_0\in\Phi(\mathbb{Q}_p^n)$. We will look for a solution belonging to $\Phi(\mathbb{Q}_p)$ and radial in $t$, for each $x\in\mathbb{Q}_p^n$, and belonging to $\Phi(\mathbb{Q}_p^n)$ in $x$, for each $t\in\mathbb{Q}_p$.

Let us consider the Cauchy problem \eqref{1}-\eqref{2}. We take the Fourier transform with respect to $x$ on both sides of \eqref{1} and obtain
\begin{equation}\Label{3}
\mathcal{F}_{x\to \xi}D^{\alpha}_{|t|_p}u(|t|_p,\xi)-|\xi|_p^{\beta}\mathcal{F}_{x\to\xi}u(|t|_p,\xi)=0,\;(t,\xi)\in\mathbb{Q}_p\times\mathbb{Q}_p^n.
\end{equation}

Suppose that the function $v(|t|_p,\xi)$, $(t,\xi)\in\mathbb{Q}_p\times\mathbb{Q}_p^n$ is a solution of the equation \eqref{3}. Then the function $u=\mathcal{F}_{\xi\to x}^{-1}v$ is the solution of \eqref{1}. Hence we will look for the solution of \eqref{3}.

We denote $\hat{u}(|t|_p,\xi)\coloneqq\mathcal{F}_{x\to\xi}u(|t|_p,\xi)$, $(t,\xi)\in\mathbb{Q}_p\times\mathbb{Q}_p$. The equation \eqref{3} is now equivalent to
\begin{equation}\Label{8}
D_{|t|_p}^{\alpha}\hat{u}(|t|_p,\xi)=|\xi|_p^{\beta}\hat{u}(|t|_p,\xi),\;(t,\xi)\in\mathbb{Q}_p\times\mathbb{Q}_p^n.
\end{equation}

The problem \eqref{8} is the radial eigenfunction problem of the operator $D_{|t|_p}^{\alpha}$. We denote $|\xi|_p=p^N$, $N\in\mathbb{Z}$. It follows from Proposition \ref{Pr1} that the solutions of the equation \eqref{8} have the form
\begin{equation}\Label{12.5}
\hat{u}_N(|t|_p,\xi)=\left\{
\begin{array}{ll}
C_N(\xi)p^{KN}(1-\dfrac{1}{p}),\;&|t|_p\leq p^{-KN};\\
-C_N(\xi)p^{KN-1},\;&|t|_p=p^{-KN+1};\\
0,\;&|t|_p\geq p^{-KN+2}.
\end{array}
\right.
\end{equation}

It follows from the initial condition \eqref{2} that
\begin{equation}\Label{13}
\hat{u}_N(0,\xi)=\hat{u}_0(\xi),\;\xi\in\mathbb{Q}_p^n.
\end{equation}
By substituting \eqref{12.5} into \eqref{13}, we get
$$\hat{u}_N(0,\xi)=C_N(\xi)p^{KN}(1-\frac{1}{p})=\hat{u}_0(\xi),\;\xi\in\mathbb{Q}_p^n.$$
Hence
$$C_N(\xi)=p^{-KN}(1-\frac{1}{p})^{-1}\hat{u}_0(\xi),\;\xi\in\mathbb{Q}_p^n.$$

Finally, we obtain the following solution of the equation \eqref{8}
\begin{equation}\Label{14}
\hat{u}_N(|t|_p,\xi)=
\left\{
\begin{array}{ll}
\hat{u}_0(\xi),\;&|t|_p\leq p^{-KN};\\
-\dfrac{1}{p-1}\hat{u}_0(\xi),\;&|t|_p=p^{-KN+1};\\
0,\;&|t|_p\geq p^{-KN+2},
\end{array}
\right.\;(t,\xi)\in\mathbb{Q}_p\times\mathbb{Q}_p^n,\;|\xi|_p=p^N.
\end{equation}

Denote
\begin{equation}\Label{15}
b_N(|t|_p,|\xi|_p)=
\left\{
\begin{array}{ll}
1,\;&|t|_p\leq p^{-KN};\\
-\dfrac{1}{p-1},\;&|t|_p=p^{-KN+1};\\
0,\;&|t|_p\geq p^{-KN+2},
\end{array}
\right.\;(t,\xi)\in\mathbb{Q}_p\times\mathbb{Q}_p^n,\;|\xi|_p=p^N.
\end{equation}
or, equivalently,
\begin{equation}\Label{15.5}
b_N(|t|_p,|\xi|_p)=
\left\{
\begin{array}{ll}
1,\;&|t|_p\leq |\xi|_p^{-K};\\
-\dfrac{1}{p-1},\;&|t|_p=p|\xi|_p^{-K};\\
0,\;&|t|_p\geq p^{2}|\xi|_p^{-K},
\end{array}
\right.\;(t,\xi)\in\mathbb{Q}_p\times\mathbb{Q}_p^n,\;|\xi|_p=p^N.
\end{equation}

Since we defined the functions $\widehat{u}_N$ and $b_N$ on each sphere $S_N^n$, we can define $\hat{u}$ and $b$ on the whole space $\mathbb{Q}_p^n$:
\begin{equation}\Label{16}
\hat{u}(|t|_p,\xi)\restriction_{S_N^n}=b_N(|t|_p,|\xi|_p)\hat{u}_0(\xi),\;t\in\mathbb{Q}_p,\xi\in S_N,
\end{equation}
\begin{equation}\Label{16.5}
b(|t|_p,|\xi|_p)\restriction_{S_N^n}=b_N(|t|_p,|\xi|_p),\;t\in\mathbb{Q}_p,\xi\in S_N.
\end{equation}

Since $\hat{u_0}\in\Psi(\mathbb{Q}_p)$, $\hat{u}(0)=0$, and it follows from local constancy of $\hat{u}$ that it vanishes on some neighbourhood of the origin. It follows from \eqref{16} that $\hat{u}\in\Psi(\mathbb{Q}_p)$ in $\xi$, so that $u\in\Phi(\mathbb{Q}_p)$ in $x$. In addition, the functions $u,\hat{u}$ belong to $\mathcal{D}(\mathbb{Q}_p)$ uniformly in the sense of support and local constancy with respect to $x$. Therefore, the interchanging of operations is permitted (in particular, the order of applying the Fourier transform $\mathcal{F}_{\xi\to x}$ and the operator $D_{|t|_p}^{\alpha}$ in the equation \eqref{3} can be reversed).

Using the inverse Fourier transform, we obtain the following solution of the equation \eqref{1}
\begin{equation}\Label{17}
u(|t|_p,x)=(\mathcal{F}_{\xi\to x}^{-1}\hat{u})(|t|_p,x)=\left((\mathcal{F}_{\xi\to x}^{-1}b)\ast u_0\right)(|t|_p,x),\;(t,x)\in\mathbb{Q}_p\times\mathbb{Q}_p^n,
\end{equation}
which belongs to $\Phi(\mathbb{Q}_p^n)$ with respect to $x$ for each $t\in\mathbb{Q}_p$ and satisfies the initial condition \eqref{2}.

It is clear from Proposition \ref{Pr1} and Proposition \ref{Pr2} that if the condition \eqref{2.5} does not hold, than the equation \eqref{1} has only a zero solution, and if $u_0$ is not identically zero, than the Cauchy problem \eqref{1}-\eqref{2} has no solutions.

\end{proof}

\section{Explicit Formula for the Solution}

Let us calculate the explicit expression for the function $(\mathcal{F}_{\xi\to x}^{-1}b)$. We have
\begin{equation}\Label{18'}
(\mathcal{F}_{\xi\to x}^{-1}b)(|t|_p,|x|_p)=\int\limits_{\mathbb{Q}_p^n}\chi_p(-\xi\cdot x)b(|t|_p,|\xi|_p)d\xi=\sum\limits_{j=-\infty}^{\infty}b(|t|_p,p^j)\int\limits_{S_j^n}\chi_p(-\xi\cdot x)d\xi.
\end{equation}

Denoting $|x|_p=p^M$ and using the formula \eqref{3.2}, we have
\begin{equation}\Label{18}
\begin{split}
&(\mathcal{F}_{\xi\to x}^{-1}b)(|t|_p,|x|_p)=(1-p^{-n})\sum\limits_{j=-\infty}^{-M}b(|t|_p,p^j)p^{jn}-b(|t|_p,p^{-M+1})p^{-Mn}=\\
&=(1-p^{-n})\sum\limits_{j=0}^{+\infty}b(|t|_p,p^{-j-M})p^{(-j-M)n}-b(|t|_p,p^{-M+1})p^{-Mn}=\\
&=(1-p^{-n})|x|_p^{-n}\sum\limits_{j=0}^{+\infty}p^{-jn}b(|t|_p,p^{-j}|x|_p^{-1})-|x|_p^{-n}b(|t|_p,p|x|_p^{-1}).
\end{split}
\end{equation}

Denoting $|t|_p=p^L$, $|x|_p=p^M$, $L,M\in\mathbb{Z}$, \eqref{18} is now equivalent to
\begin{equation}\Label{19}
\mathcal{F}_{\xi\to x}^{-1}b(p^L,p^M)=(1-p^{-n})p^{-Mn}\sum\limits_{j=0}^{+\infty}p^{-jn}b(p^L,p^{-M-j})
-p^{-Mn}b(p^L,p^{-M+1}).
\end{equation}
For $j\geq0$ from \eqref{15} we have
\begin{equation}\Label{20}
b(p^L,p^{-M-j})=b_{-M-j}(p^L,p^{-M-j})=
\left\{
\begin{array}{ll}
1,\;&p^L\leq p^{K(M+j)};\\
-\dfrac{1}{p-1},\;&p^L=p^{K(M+j)+1};\\
0,\;&p^L\geq p^{K(M+j)+2},
\end{array}
\right.
\end{equation}
or, equivalently,
\begin{equation}\Label{20'}
b(|t|_p,p^{-j}|x|_p^{-1})=b(|t|_p,|p^jx^{-1}|_p)=
\left\{
\begin{array}{ll}
1,\;&|t|_p\leq p^{jK}|x|_p^K;\\
-\dfrac{1}{p-1},\;&|t|_p=p^{jK+1}|x|_p^K;\\
0,\;&|t|_p\geq p^{jK+2}|x|_p^K,
\end{array}
\right.
\end{equation}

\begin{equation}\Label{21}
b(p^L,p^{-M+1})=b_{-M+1}(p^L,p^{-M+1})=
\left\{
\begin{array}{ll}
1,\;&p^L\leq p^{K(M-1)};\\
-\dfrac{1}{p-1},\;&p^L=p^{K(M-1)+1};\\
0,\;&p^L\geq p^{K(M-1)+2},
\end{array}
\right.
\end{equation}
or, equivalently,
\begin{equation}\Label{21'}
b(|t|_p,p|x|_p^{-1})=b(|t|_p,|p^{-1}x^{-1}|_p)=
\left\{
\begin{array}{ll}
1,\;&|t|_p\leq p^{-K}|x|_p^K;\\
-\dfrac{1}{p-1},&|t|_p=p^{-K+1}|x|_p^K;\\
0,\;&p^L\geq p^{-K+2}|x|_p^K.
\end{array}
\right.
\end{equation}

We consider the following cases.

\underline{Case 1}. $|t|_p\leq|x|_p^K$, or $\dfrac{L}{K}\leq M\Longleftrightarrow L\leq KM$. Then from \eqref{20} and \eqref{21} we have
\begin{equation}\Label{22}
\mathcal{F}_{\xi\to x}^{-1}b(p^L,p^M)=(1-p^{-n})p^{-nM}\sum\limits_{j=0}^{+\infty}p^{-j}-p^{-nM}b(p^L,p^{-M+1}).
\end{equation}

\textit{Case 1.a)}. $|t|_p\leq p^{-K}|x|_p^K$, or $L\leq K(M-1)$. Then
\begin{equation}\Label{23}
\mathcal{F}_{\xi\to x}^{-1}b(p^L,p^M)=(1-p^{-n})p^{-nM}\sum\limits_{j=0}^{+\infty}p^{-nj}-p^{-nM}=(1-p^{-n})p^{-nM}\frac{1}{1-p^{-n}}-p^{-nM}=0.
\end{equation}

\textit{Case 1.b)}. $|t|_p=p^{-K+1}|x|_p^K$, or $L=K(M-1)+1$. Then
\begin{equation}\Label{24}
\mathcal{F}_{\xi\to x}^{-1}b(p^L,p^M)=p^{-nM}-p^{-nM}(-\frac{1}{p-1})=p^{-nM}\frac{p}{p-1}=\frac{p^{-nM+1}}{p-1}=\frac{p}{p-1}|x|_p^{-n}.
\end{equation}

\textit{Case 1.c)}. $|t|_p\geq p^{-K+2}|x|_p^K$, or $L\geq K(M-1)+2$. Then
\begin{equation}\Label{25}
\mathcal{F}_{\xi\to x}^{-1}b(p^L,p^M)=p^{-nM}=|x|_p^{-n}.
\end{equation}

\underline{Case 2}. $|t|_p=p|x|_p^K$, or $\dfrac{L-1}{K}-M=0\Longleftrightarrow L=KM+1$. Then from \eqref{20} and \eqref{21} we have
\begin{equation}\Label{26}
\begin{split}
\mathcal{F}_{\xi\to x}^{-1}b(p^L,p^M)=(1-p^{-n})p^{-nM}\sum\limits_{j=1}^{+\infty}p^{-nj}+(1-p^{-n})p^{-nM}(-\frac{1}{p-1})=\\
=(1-p^{-n})p^{-nM}\frac{p^{-n}}{1-p^{-n}}-\frac{1-p^{-n}}{p-1}=p^{-nM}\frac{p^{-n+1}-1}{p-1}.
\end{split}
\end{equation}

\underline{Case 3}. $|t|_p\geq p^2|x|_p^K$, or $\dfrac{L-2}{K}-M\geq0\Longleftrightarrow L\geq KM+2$.

\textit{Case 3.a)} $\dfrac{L-1}{K}-M$ is an integer. Then
\begin{equation}\Label{27}
\begin{split}
\mathcal{F}_{\xi\to x}^{-1}b(p^L,p^M)=(1-p^{-n})p^{-nM}\sum\limits_{j=[\frac{L}{K}-M]}^{+\infty}p^{-nj}+(1-p^{-n})p^{-nM}p^{-n(\frac{L-1}{K}-M)}(-\frac{1}{p-1})=\\
=p^{-n[\frac{L}{K}]}-\frac{1-p^{-n}}{p-1}p^{-n\frac{L-1}{K}}.
\end{split}
\end{equation}

\textit{Case 3.b)} $\dfrac{L-1}{K}-M$ is not an integer. Then
\begin{equation}\Label{28}
\mathcal{F}_{\xi\to x}^{-1}b(p^L,p^M)=p^{-n[\frac{L}{K}]}.
\end{equation}

\section{Finite Domain of Dependence}

\begin{proposition}\Label{prH}
Suppose that the initial function $u_0$ is such that $u_0(t,x)=0$ outside some ball $B_N^n$, $N\in\mathbb{Z}$. Then the solution $u$ of the problem \eqref{1}-\eqref{2} satisfies $u(|t|_p,x)=0$ for $x\in\mathbb{Q}_p^n\setminus B_N^n$, and for all $|t|_p<p^L$ for some $L\in\mathbb{Z}$.
\end{proposition}

\begin{proof}
From \eqref{17} we have
$$u(|t|_p,x)=\left((\mathcal{F}_{\xi\to x}^{-1}b)\ast u_0\right)(|t|_p,x)=\int\limits_{\mathbb{Q}_p^n}(\mathcal{F}_{\xi\to x}^{-1}b)(|t|_p,|x-y|_p)u_0(y)dy=$$
$$=\int\limits_{B_N^n}(\mathcal{F}_{\xi\to x}^{-1}b)(|t|_p,|x-y|_p)u_0(y)dy.$$
Suppose that $|x|_p>p^N$. Then for $y\in B_N^n$ by the triangle inequality (which is an equality in this case since $|x|_p\neq|y|_p$) we have $|x-y|_p=\max(|x|_p,|y|_p)=|x|_p$. Then $p^{-K}|x|_p^K>p^{-K}p^{NK}$. If we choose $L=p^{NK-K}$, then for all $t$, $|t|_p\leq p^L$ we have
$$|t|_p\leq p^L\leq p^{NK-K}=p^{-K}|x|_p^K.$$
Recalling the explicit formula for the function $\mathcal{F}_{\xi\to x}^{-1}b$, it follows from \eqref{23} that for such $x$ and $t$ $(\mathcal{F}_{\xi\to x}^{-1}b)(|t|_p,|x-y|_p)=0$, when $y\in B_N^n$. Hence $u(|t|_p,x)=0$.

\end{proof}

\section{Generalized Solutions and Uniqueness}

We will need the following property of the operator $D^{\alpha,n}$ of a distribution that corresponds to the bounded locally constant function, proved in \cite{K2008}.

\begin{lemma}[Lemma 1, \cite{K2008}]\Label{lemma2}
If $u$ is a bounded locally constant function on $\mathbb{Q}_p^n$, then the distribution $D^{\alpha,n}u\in\Phi(\mathbb{Q}_p^n)$ coincides with the function
\begin{equation}\Label{34}
(D^{\alpha,n}u)(x)=\dfrac{1-p^{\alpha}}{1-p^{-\alpha-n}}\int\limits_{\mathbb{Q}_p}|y|_p^{-\alpha-n}\left[u(x-y)-u(x)\right]dy.
\end{equation}
\end{lemma}

We consider the problem \eqref{1}-\eqref{2} in the class of generalized functions, radial in $t$.

Denote by $\Phi'(\mathbb{Q}_p,\Phi'(\mathbb{Q}_p^n))$ the set of distributions over the test function space $\Phi(\mathbb{Q}_p)$, with values in $\Phi'(\mathbb{Q}_p^n)$.

\begin{theorem}\Label{th2}
Let $u\in\Phi'(\mathbb{Q}_p,\Phi'(\mathbb{Q}_p^n))$ be a generalized solution of the equation \eqref{1}, that is
$$\langle\langle u,D_t^{\alpha}\varphi_1\rangle,\varphi_2\rangle=\langle\langle u,\varphi_1\rangle,D_x^{\beta,n}\varphi_2\rangle,$$
for any $\varphi_1\in\Phi(\mathbb{Q}_p)$, $\varphi_2\in\Phi(\mathbb{Q}_p')$. If $u$ is radial in $t$, then $u\in\mathcal{D}(\mathbb{Q}_p,\Phi'(\mathbb{Q}_p^n))$. If, in addition, $u(0,x)=0$, then $u(t,x)\equiv0$.
\end{theorem}

\begin{proof}
Denote by $\widehat{u}(t,\cdot)$ the Fourier transform of $u$ in the variable $x$. For any $\psi\in\Psi(\mathbb{Q}_p^n)$ we have
$$D_{|t|_p}^{\alpha}\langle\widehat{u}(t,\cdot),\psi\rangle=\langle|\xi|_p^{\beta}\widehat{u}(t,\xi),\psi(\xi)\rangle.$$
If $\supp\psi\subset S_N$, $N\in\mathbb{N}$, then
$$D_{|t|_p}^{\alpha}\langle\widehat{u}(t,\cdot),\psi\rangle=p^{\beta N}\langle\widehat{u}(t,\cdot),\psi(\cdot)\rangle.$$

By Proposition \ref{Pr2}, the function $\langle \widehat{u}(t,\cdot),\psi\rangle$ has the form \eqref{3.12} for some $c\in\mathbb{C}$. If $\psi\in\Psi(\mathbb{Q}_p^n)$, then $\psi$ is a sum of a finite number of functions supported on spheres $S_N^n$. Taking, in particular, $\psi=\widehat{\varphi}$, $\varphi\in\Phi(\mathbb{Q}_p^n)$, we find that $\langle u(t,\cdot),\varphi\rangle$ belongs to $\mathcal{D}(\mathbb{Q}_p)$ in the variable $t$, for any $\varphi\in\Phi(\mathbb{Q}_p^n)$.

If $u(0,\cdot)=0$, then also $\widehat{u}(0,\cdot)=0$. If $\psi\in\Psi(\mathbb{Q}_p^n)$, $\supp\psi\subset S_N^n$, then $\langle\widehat{u}(t,\cdot),\psi\rangle$ has the form \eqref{3.2}, and the assumption $\widehat{u}(0,\cdot)=0$ implies the equality $c=0$, hence $\langle\widehat{u}(t,\cdot),\psi\rangle=0$, and $\widehat{u}(t,\cdot)=0$ (since $\psi$ and $N$ are arbitrary), and $u(t,\cdot)=0$.
\end{proof}

It follows from Lemma \ref{lemma2} that locally bounded constant solutions of the equation \eqref{1} are generalized solutions of the class considered in Theorem \ref{th2}. Therefore the solutions of the Cauchy problems constructed in Theorem \ref{th1} are unique in the class of radial in $t$, bounded locally constant functions.

\section{$L^1$-Estimate of the Solution}

Denote $\mathbb{Q}_p^+\coloneqq\{t\in\mathbb{Q}_p|\;|t|_p\geq1\}$.

\begin{theorem}\Label{th3}
Suppose that the conditions of Theorem \ref{th1} hold. Then the solution of the problem \eqref{1}-\eqref{2}, defined in \eqref{17}, satisfies the following estimate in $L^1(\mathbb{Q}_p^n)$ in variable $x$ for each $t\in\mathbb{Q}_p^+$
 \begin{equation}\Label{31.5}
\norm{u(|t|_p,\cdot)}_{L_1(\mathbb{Q}_p^n)}\leq p^{2n\gamma}\norm{u_0}_{L_1(\mathbb{Q}_p^n)},
\end{equation}
where $\gamma\geq\dfrac{2}{K}$ is a positive constant.
\end{theorem}

\begin{proof}
For each $t\in\mathbb{Q}_p^+$ we have
\begin{equation}\Label{30}
\begin{split}
&\norm{\mathcal{F}_{\xi\to x}^{-1}[b(|t|_p,|\xi|_p)](|t|_p,\cdot)}_{L_1(\mathbb{Q}_p^n)}=\int\limits_{\mathbb{Q}_p^n}\abs*{\mathcal{F}_{\xi\to x}^{-1}[b(|t|_p,|\xi|_p))](|t|_p,\cdot)}dx=\\
&=\int\limits_{\mathbb{Q}_p^n}\abs*{\int\limits_{\mathbb{Q}_p^n}\chi_p(-x\cdot\xi)b(|t|_p,|\xi|_p))d\xi}dx=\\
&=\int\limits_{\mathbb{Q}_p^n}\abs*{\int\limits_{B_{\gamma}^n}\chi_p(-x\cdot\xi)b(|t|_p,|\xi|_p))d\xi+\int\limits_{\mathbb{Q}_p^n\setminus B_{\gamma}^n}\chi_p(-x\xi)b(|t|_p,|\xi|_p))d\xi}dx\leq\\
&\leq\int\limits_{\mathbb{Q}_p^n}\abs*{\int\limits_{B_{\gamma}^n}\chi_p(-x\cdot\xi)b(|t|_p,|\xi|_p))d\xi}dx+\int\limits_{\mathbb{Q}_p^n}\abs*{\int\limits_{\mathbb{Q}_p^n\setminus B_{\gamma}^n}\chi_p(-x\cdot\xi)b(|t|_p,|\xi|_p)d\xi}dx,
\end{split}
\end{equation}
where $\gamma\in\mathbb{Z}$ will be chosen later.

For the integral $\int\limits_{\mathbb{Q}_p^n\setminus B_{\gamma}^n}\chi_p(-x\cdot\xi)b(|t|_p,|\xi|_p)d\xi$ we have
$$\int\limits_{\mathbb{Q}_p^n\setminus B_{\gamma}^n}\chi_p(-x\cdot\xi)b(|t|_p,|\xi|_p)d\xi=\sum\limits_{N=\gamma+1}^{+\infty}\int\limits_{S_N^n}\chi_p(-x\cdot\xi)b(|t|_p,|\xi|_p)d\xi.$$

It follows from \eqref{15.5} that $b(|t|_p,|\xi|_p)=0$, if $|\xi|_p^K\geq\dfrac{p^2}{|t|_p}$. We suppose that $|t|_p=p^L$, $L\geq0$, then $|t|_p\geq1$. Thus if $\gamma\geq\frac{2}{K}$, for $\xi\in S_N^n$, $N\geq\gamma+1$, we get $|\xi|_p\geq p^{\gamma+1}=p^{\frac{2}{K}+1}\geq p^{\frac{2}{K}}$, then $|\xi|_p^K\geq p^2\geq\dfrac{p^2}{|t|_p}$, and therefore $\int\limits_{S_N^n}\chi_p(-x\cdot\xi)b(|t|_p,|\xi|_p)d\xi=0$. In conclusion, for the chosen $\gamma$ we have
$$\int\limits_{\mathbb{Q}_p^n}\abs*{\int\limits_{\mathbb{Q}_p^n\setminus B_{\gamma}^n}\chi_p(-x\cdot\xi)b(|t|_p,|\xi|_p)d\xi}dx=0.$$

Now for the first integral in \eqref{30} we have
$$\int\limits_{\mathbb{Q}_p^n}\abs*{\int\limits_{B_{\gamma}^n}\chi_p(-x\cdot\xi)b(|t|_p,|\xi|_p)d\xi}dx=\int\limits_{B_{\gamma}^n}\abs*{\int\limits_{B_{\gamma}^n}\chi_p(-x\cdot\xi)b(|t|_p,|\xi|_p)d\xi}dx+$$
$$+\int\limits_{\mathbb{Q}_p^n\setminus B_{\gamma}^n}\abs*{\int\limits_{B_{\gamma}^n}\chi_p(-x\cdot\xi)b(|t|_p,|\xi|_p)d\xi}dx,$$
where
$$\int\limits_{B_{\gamma}^n}\abs*{\int\limits_{B_{\gamma}^n}\chi_p(-x\cdot\xi)b(|t|_p,|\xi|_p)d\xi}dx\leq\int\limits_{B_{\gamma}^n}\int\limits_{B_{\gamma}^n}|\chi_p(-x\cdot\xi)|\cdot|b(|t|_p,|\xi|_p)|d\xi dx\leq\int\limits_{B_{\gamma}^n}\int\limits_{B_{\gamma}^n}d\xi dx=p^{2n\gamma},$$
since
$|\chi_p(-x\cdot\xi)|=1$ for all $x,\xi\in\mathbb{Q}_p^n$ by definition of the character $\chi_p$, and $|b(|t|_p,|\xi|_p)|\leq1$ for all $(t,\xi)\in\mathbb{Q}_p^+\times\mathbb{Q}_p^n$ by definition of the function $b$ in \eqref{15.5}.

Finally, we have from \eqref{3.1}
$$\int\limits_{B_{\gamma}^n}\chi_p(-x\cdot\xi)b(|t|_p,|\xi|_p)d\xi\leq\int\limits_{ B_{\gamma}^n}\chi_p(-x\cdot\xi)d\xi=
\left\{
\begin{array}{ll}
p^{n\gamma},\;&|x|_p\leq p^{-\gamma};\\
0,\;&|x|_p\geq p^{-\gamma+1},
\end{array}
\right.$$
and
$$\int\limits_{B_{\gamma}^n}\chi_p(-x\cdot\xi)b(|t|_p,|\xi|_p)d\xi\geq-\frac{1}{p-1}\int\limits_{B_{\gamma}^n}\chi_p(-x\cdot\xi)d\xi=
\left\{
\begin{array}{ll}
-\dfrac{1}{p-1}p^{n\gamma},\;&|x|_p\leq p^{-\gamma};\\
0,\;&|x|_p\geq p^{-\gamma+1}.
\end{array}
\right.$$
Since $\gamma\geq\dfrac{2}{K}$ is positive, in the outer integral we have $|x|_p\geq p^{\gamma+1}\geq p^{-\gamma+1}$, so we get
$$\int\limits_{\mathbb{Q}_p^n\setminus B_{\gamma}^n}\abs*{\int\limits_{B_{\gamma}^n}\chi_p(-x\cdot\xi)b(|t|_p,|\xi|_p)d\xi}dx=0.$$

Finally, combining the estimates for all integrals in \eqref{30}, we obtain
\begin{equation}\Label{31}
\norm{\mathcal{F}_{\xi\to x}^{-1}[b(|t|_p,|\xi|_p)](|t|_p,\cdot)}_{L_1(\mathbb{Q}_p^n)}\leq p^{2n\gamma}.
\end{equation}
The inequality for convolutions gives us the estimate for the norm of the solution
\begin{equation}\Label{31}
\norm{u(|t|_p,\cdot)}_{L_1(\mathbb{Q}_p^n)}\leq p^{2n\gamma}\norm{u_0}_{L_1(\mathbb{Q}_p^n)}.
\end{equation}
\end{proof}

\section*{Acknowledgement}
The authors acknowledge the financial support by the National Research Foundation of Ukraine (Project number 2023.03/0002).


\begin{thebibliography}{}

\bibitem{A}\Label{A}
S. Albeverio, A. Yu. Khrennikov, and V. M. Shelkovich, \textit{Theory of p-Adic Distributions}, Cambridge University Press, 2010.

\bibitem{G-C}\Label{G-C}
J. Galeano-Peñaloza, O. F. Casas-Sánchez, Radial solutions of a pseudo-differential equation associated with a $p$-adic non-local ultradiffusion operator, \textit{J. Pseudo-Differ. Oper. Appl.}, \textbf{15(2)} (2024).

\bibitem{K2008}\Label{K2008}
A.N. Kochubei, A Non-Archimedean Wave Equation, \textit{Pacif. J. Math}, \textbf{235(2)} (2008), 245-261.

\bibitem{K2014}\Label{K2014}
A.N. Kochubei, Radial Solutions of non-Archimedean Pseudo-differential Equations, \textit{Pacif. J. Math.}, \textbf{269} (2014), 355-369.

\bibitem{K2001}\Label{K2001}
A. N. Kochubei, \textit{Pseudo-Differential Equations and Stochastics over Non-Archimedean Fields}, Marcel Dekker, New York, 2001.

\bibitem{K2023}\Label{K2023}
A.N. Kochubei, The Vladimirov-Taibleson operator: inequalities, Dirichlet problem, boundary Hölder regularity, \textit{J. Pseudo-Differ. Oper. Appl.}, \textbf{14} (2023), No.2, Paper 31, 27 pp.

\bibitem{VVZ}\Label{VVZ}
V. S. Vladimirov, I. V. Volovich and E. I. Zelenov, \textit{p-Adic Analysis and Mathematical
Physics}, World Scientific, Singapore, 1994.

\bibitem{Z}\Label{Z}
W.A. Zúñiga-Galindo, \textit{Pseudodifferential equations over non-Archimedean spaces}, Lecture Notes Math 2174, 2016.

\end{thebibliography}
\end{document}